\documentclass[11pt]{amsart}

\usepackage{amsmath}
\usepackage{amssymb}
\usepackage{amsthm}
\usepackage{amscd, amsfonts}
\usepackage[dvips]{epsfig}
\usepackage{verbatim}
\usepackage{epsf}

\usepackage{hyperref}

\usepackage{mathrsfs}

\newcommand{\bj}{\bar{j}}

\newcommand{\bl}{\bar{l}}

\newcommand{\al}{\alpha}
\newcommand{\bal}{\bar{\alpha}}
\newcommand{\be}{\beta}
\newcommand{\bbe}{\bar{\beta}}
\newcommand{\ga}{\gamma}

\newcommand{\De}{\Delta}
\newcommand{\de}{\delta}
\newcommand{\bde}{\bar{\delta}}
\newcommand{\eps}{\epsilon}
\newcommand{\ka}{\kappa}

\newcommand{\bla}{\bar{\lambda}}

\newcommand{\m}{\mu}

\newcommand{\n}{\nu}
\newcommand{\bn}{\bar{\nu}}

\newcommand{\tta}{\theta}

\newcommand{\si}{\sigma}

\newcommand{\ta}{\tau}

\newcommand{\ro}{\varrho}

\newcommand{\fa}{\mathfrak{a}}

\newcommand{\fc}{\mathfrak{c}}

\newcommand{\fH}{\mathfrak{H}}

\newcommand{\fraki}{\mathfrak{i}}

\newcommand{\fK}{\mathfrak{K}}

\newcommand{\fw}{\mathfrak{w}}

\newcommand{\fz}{\mathfrak{z}}

\newcommand{\bfz}{\bar{\mathfrak{z}}}

\newcommand{\calO}{\mathcal{O}}

\newcommand{\calS}{\mathcal{S}}

\newcommand{\om}{\omega}

\newcommand{\bbC}{\mathbb{C}}

\newcommand{\bbR}{\mathbb{R}}

\newcommand{\bbN}{\mathbb{N}}

\newcommand{\tr}{\operatorname{tr}}

\newcommand{\del}{\operatorname{\partial}}

\newcommand{\wge}{\wedge}

\theoremstyle{plain}
\newtheorem{theorem}{Theorem}[section]
\newtheorem{lemma}[theorem]{Lemma}
\newtheorem{proposition}[theorem]{Proposition}

\theoremstyle{definition}
\newtheorem{remark}[theorem]{Remark}
\newtheorem{definition}[theorem]{Definition}

\numberwithin{equation}{section}
\author[Aleyasin]{S. Ali Aleyasin}
\address{Dept. Of Pure Mathematics- Faculty of Mathematics\\
University of Waterloo- 200- University Ave. W. \\
N2L 3G1, Canada
 }
\email{aaleyasin@uwaterloo.ca}

\thanks{Most of this work was compiled when the author was a CIRGET post-doc and visiting IPM in Tehran in 2014-15.
I am grateful to Prof. Xiu-Xiong Chen for introducing me to  this circle of problems and for his constant encouragement. 
I should like to thank Professors Apostolov, P-F. Guan, S: Lu, and Rochon with whom I have had various discussions during the 
completion of the present work.}

\title[Calabi problem and edge-cone metrics]{Calabi problem for manifolds with edge-cone singularities}

\begin{document}

\begin{abstract}
In this note, we propose a new approach to  solving the Calabi problem on manifolds with edge-cone singularities of 
prescribed angles along complex hypersurfaces.
It is shown how the classical approach of Aubin-Yau in derving {\it a priori} estimates for the complex hessian can be made to 
work via adopting a \emph{good reference metric} and studying equivalent equations with different referrence metrics.
This further allows extending much of the methods used in the smooth setting to the edge setting.
These results generalise to the case of multiple hypersufaces with possibly normal crossing.

53C55, %Hermitian and Kählerian manifolds 
35J60 %non-linear elliptic pde

\end{abstract}

\maketitle

\section{Introduction}

%[[Blocki's gradient estimate? Along with Chen-He etc. in a short separate note, without talking too much about the proofs, 
%just `the following theorems hold'.]].

%Solution to the Calabi problem has been the subject of serious study in the past few years. 
%In particular, finding K\"ahler-Einstein metrics has received attention thanks to its connections to certain problems of algebraic geometry.
The study of problems around cone-edge singularities, in particular, the problems around finding K\"ahler-Einstein metrics prescribed edge-cone behaviour,
 has received quite a bit of attention in the past years. 
One reason is the key r\^{o}le such metrics play  in the approach taken by Chen, Donalsdon, and Sun in \cite{ch-do-su-1, ch-do-su-2, ch-do-su-3}, (and also other attempts for solving the same problem in \cite{ti}), in proving the relation between \(K\)-stability and the existence of K\"ahler-Einstein metrics on Fano manifolds. 
In the present work we shall show how one can use the geometry of  edge-cone manifolds by constructing edge metrics with curvature bounded from below to obtain the estimates needed for solving the Calabi problem.
Other than finding K\"ahler-Einstein metrics, this approach allows prescribing a wide class of Ricci forms.
Since in most constructions and proofs it is straightforward to see how they should be modified for the case of divisors with possibly normal crossing, in the rest of this work, in order to keep the statements and proofs clearer, we confine ourselves to the case of one smooth hypersurface.

By K\"ahler metrics with \emph{edge} or \emph{edge-cone} singularities we mean a K\"ahler metric with conical singularity along a complex hypersurface, that is, a metric which asymptotically resembles a cone on \(\bbC\) of total angle \(2 \pi \ta\) in the directions normal to the hypersurface, and  is smooth in the tangential directions.
Examples of such metrics were already known as they arise as orbifold metrics.
More generally, one may construct such metrics as follows. Let \((M^n,\om_0)\) K\"ahler manifolds, where \(\om_0\) is smooth. 
Assume that \(D^{n-1} \subset M^n\) is a complex hypersurface and that \(s\) is a holomorpic section of an hermitian line bundle \((L,h)\) which vanishes of order zero along \(D\). Then, the following metric 
\begin{equation}
\label{cone-reference}
	\om_\ta := \om_0 +\fa dd^c |s|_h^{2\ta}
\end{equation}
is an edge-cone metric along \(D\) of angle \(\ta\) when \(0<\fa \ll 1\).
This statement, along with the rest of results can be generalised to the case wherein \(D\) consistes of a union of irreducible divisors, \(D_j\), with at most normal crossing. 
Indeed, to the best of the author's knowledge, in the works prior to the present work the metric \(\om_\ta\), thus defined,  has been taken as the conical background metric in order to do the analysis.

Various approaches to the study of the Calabi problem on manifolds with edge-cone singularities have proved effective. 
Ricci-flat edge-cone K\"ahler metrics  were proved to exist under suitable topological conditions by Brendle in \cite{br} provided that \(\ta \leq {1 \over 2}\).
There, the classical method used by Aubin and Yau in \cite{au,ya} was used.
The fact that the classical calculations can be adapted to this situation   is due to the fact that the reference metric \(\om_\ta\) has bounded curvature away from the divisor \(D\) when \(\ta \leq {1 \over 2}\).
More specifically, when \(\ta > {1 \over 2}\), 
the curvature of the reference metric \eqref{cone-reference} might become unbounded from below. 
This possible lack of lower bound in the curvature has been one main obstacle in deriving laplacian estimates and higher order regularity results since the approach of Aubin and Yau 
for deriving estimates on the complex hessian depends on the existence of a lower bound on the  bisectional curvature of the referrence metric. 
Further, the lack of of a lower bound on the Ricci curvature of the reference metric means the lack of a lower bound on the laplacian of the Ricci potential, a yet another quantity the finiteness of whose lower bound is needed in deriving the laplacian estimate for the potential in the approach of Aubin-Yau.

Another work, which covers the case of larger angles as well is that of Jeffres, Mazzeo and Rubinstein \cite{je-ma-ru}.
There, in order to derive the laplacian estimate in the absence of a lower bound on curvature of \(\om_\ta\) a corollary of the Chern-Lu inequality has been used which, rather than the lower bound of the curvature of the background metric requires an upper bound on it.
Along with the observation that the curvature of \(\om_\ta\) is always bounded from above, the existence of K\"ahler-Einstein metrics are proved in \cite{je-ma-ru}.
Using a different approach to the Chern-Lu inequality, X.-X. Chen \emph{et al} have also derived the laplacian estimate in the Fano case in \cite{ch-do-su-2}.
This has been further refined by Yao in \cite{yao}.

In a different direction, in \cite{gu-pa} one finds  a clever adaptation of the classical calculations to the edge-cone setting by using an auxiliary function to control the behaviour of the possibly unbounded curvature terms in the study of the complex Monge-Amp\`ere equation.
The interested reader is referred to Chapter 7 of the survey article by Rubinstein \cite{rub}, wherein 
an extensve treatment of the laplacian estimates in the works mentioned above can be found.

However, one core idea in the present work is the observation that indeed the Pogorelov-Aubin-Yau approach can be used with very little modification if one observes that within the same edge-cone cohomology class there always exist  metrics with  lower bound on their bisectional curvature, and that with respect to this reference metric, the Ricci potential has the correct behaviour.
This allows us to show the following.

\begin{theorem}
  \label{thm-1-1}
	(a) [Calabi's first problem] Let \(\tilde{\ro} \in c_1(M) - (1- \ta) c_1(L) \in H^{1,1}(M,\bbR) \) be a closed 
	real (1,1)-form which is of class \(C^\al\) on \(M \sim D\),
	such that \((\tr_{\om_\ta} \tilde{\ro})^+ = o(|s|^{-2\ta})\) for some \(\eps>0\), and in the vicinity that \(\tilde{\ro}\) is generated by a local
	potential of the H\"older class \(C^\al_\tta\) for some exponent \(\al\).
Then, there is a potential \(\phi\), which belongs to the class \(C^{2,\tta}_\ta\) for some exponent \(\tta\), determined uniquely up to a constant, such that 
\begin{equation}
\label{prescribe-ricci}
	\ro(\om_\ta + dd^c \phi) = \tilde{\ro} +2 \pi (1 - \ta) [D] 
\end{equation}

	(b) [Calabi's second problem]
Assume that \(c(M)- (1 - \ta)c_1(L) = \m [\om_0]\), where, either i) \(\m \leq 0\), or ii) we have an \(L^\infty\) bound on the potential \(\phi\).
Then, there is a potential, belonging to the class \(C^{2,\tta}_\ta\) for some exponent \(\tta\), such that 
\begin{equation}
\label{twisted-ke}
	\ro( \om_{\ta,\phi}) = \m \om_{\ta,\phi} + 2 \pi(1 - \ta) [D]
\end{equation}
for \(\om_{\ta,\phi}=\om_\ta + dd^c \phi\).
Further, the potential \(\phi\) is unique up to addition of a constant  when \(\m = 0\),
and is unique otherwise.

\end{theorem}

As in the smooth case, the solution of the Calabi problem relies on solving a complex Monge-Amp\`ere equation, but this time with an edge K\"ahler reference metric.
This connection, which calls for a bit more careful handling than in the smooth case, will be clarified in \textsection \ref{mise-en-equation}.

\begin{theorem}
\label{thm-1-2}
	Assume that the metric \(\om_\ta\) is as defined in \eqref{cone-reference}, and that for some exponent \(\al \in (0,1)\),  \(f \in C^\al_\ta\)
	is a function such that \( (\De f)^- = o( |s|_h^{-2 \ta})\).
	Then, there exists a solution \(\phi\) to the following equation:
	\begin{equation}\label{eq-1-4}
		(\om_{\ta} + dd^c \phi)^n = e^f \om_{\ta}^n; \text{  } \int e^f 		\om_{\ta}^n = \int \om_{\ta}^n
	\end{equation}
	which is unique up to a constant and belongs to the edge-cone H\"older space \(C^{2,\tta}_\ta\) for some exponent \(\tta\).
\end{theorem}

In order to solve the equation, we also use idea of approximating the edge-cone metric by a family of smooth metrics similar to what is done in \cite{ca-gu-pa, gu-pa}. 
An important observation that allows deriving estimates independent of the upper bound on the scalar curvature of the referrence metric is also borrowed from an earlier work of P\v{a}un \cite{pa}.

The proof of the theorems above relies on the following proposition which allows us to explicitly construct a good reference metric.
%Let \(\calH_\ta\) be the space of all conical K\"ahler metrics of a given angle \(\ta\) along a fixed divisor in one fixed K\"ahler class.

\begin{proposition}
\label{good-metric-pro}
		Let \(\om_0\) be a smooth K\"ahler metric and let \(\ta \in (0,1)\) and \(\ta' \in (\ta,1)\) be two real numbers.
	Then, for sufficiently small positive constant \(\fc > 0\), the following (1,1)-form 
	\begin{equation}	\label{good-metric-eq-0}
		\tilde{\om} := \om_\ta - \fc dd^c |s|_h^{2 \ta'}=
		 \om_0 + \fa dd^c |s|_h^{2 \ta} - \fc dd^c |s|_h^{2 \ta'}
	\end{equation} 
	is an edge-cone K\"ahler metric of cone angle \(\ta\), equivalent to the following edge-come metric:
	\begin{equation}
		\om_{\ta}:= \om_0 + \fa dd^c |s|_h^{2  \ta}
	\end{equation} 
	in such a way that the curvature of \(\tilde{\om}\) is bounded from below.
	Further, 
	the parameter \(\fc\)  can be chosen to be sufficiently small so that the metrics \(\om_\ta\) and \(\tilde{\om}_\ta\) are arbitrarily close 
	with respect to the H\"older norm \(C^\tta_\ta\) .
\end{proposition}

Having proved the proposition above, we have to  study how the equation and its components, in particular, the Ricci potential, transform under this change of the reference metric and then prove that we can indeed derive the estimates.
The proof of this proposition and the study of how the equation transform under the change of the reference metric is the subject of \textsection \ref{good-metric-sec}.

%The condition on the lower bound of the curvature of the reference metric appears not only in the works of Aubin and Yau, but also in other estimates derived later on, for example in \cite{bl-1, bl-2}.
Besides opening way for estimates previously known in the smooth case, which we shall hopefully explore elsewhere, one advantage of this approach is that this method is based on the geometry of the edge-cone metrics
and we, therefore, hope that such observations about conical metrics will be of interest in their own right. 
Further, this approach  allows a wide class of preassigned Ricci forms to be realised by edge-cone K\"ahler metrics.

\section{Edge-cone functional spaces, Linear elliptic theory }

In this section, we shall introduce the notation and basic concepts we shall frequently make use of in the rest of this note.
For the definition of spaces we follow
\cite{do},
where the linear elliptic theory for edge-cone metrics is developed.

Let us for the sake of clarity work on \(\bbC^n\) and assume that the edge-cone singularity occurs along the divisor \(\sum_{j=1}^k (1-\ta_j) [\fz_j = 0]\).
The edge-cone K\"ahler form we consider as our model on \(\bbC^n\) is the following:
\begin{equation}
\label{cone-model}
	\om_{\fK} :=
	dd^c \big(
	\sum_{j=1}^k |\fz_j|^{2\ta_j} + \sum_{j=k+1}^n |\fz_j|^2
	\big)
	=
	  \sum_{j=1}^k  \fraki \ta_j^2 \vert \fz_j \vert^{2\ta_j - 2}  d \fz_j \wge d \bfz_j
	+ 
	\sum_{j=k+1}^n \fraki d \fz_j \wge d \bfz_j
\end{equation}

To keep our notation simpler, in the following definitions we shall only state the case of a single divisor along \([\fz_1 = 0]\). 
The definitions can be extended to the case of multiple divisors with possible normal crossing in the obvious way.

\begin{definition}
Consider the K\"ahler space \(\bbC_\ta \times \bbC^{n-1}\).
For a given function 
\(f(\fz^1,..., \fz^n)\) 
define the associated function 
\(\tilde{f}(\xi, ... , \fz^n)\)
where \(\xi = |\fz_1|^{\ta - 1} \fz_1\).
Then, \(f\) is said to be of class \(C^\al_\ta\) provided that \(\tilde{f} \in C^\al\).
\end{definition}

There is another change of variable which we shall use which
is compatible with the picture one has of in the case of conical angle 
\(\ta={1 \over p}\), for \(p \in \bbN\), 
on \(\bbC_{1 \over p}\),
which can be \emph{uniformised} via the map 
\begin{eqnarray}
	\bbC_{1 \over p} &\to& \bbC \nonumber \\
	\fz &\mapsto& \fz^{1 \over p} =: \fw
\end{eqnarray}
The advantage of this transformation, unlike transformation \(\fz \mapsto \xi\), is that  change of variable to \(\fw\) is a -local- bi-holomorphism, and hence we can calculate  geometric quantities such as curvature in \(\zeta\).
It is important to observe that \(\fw\) and and \(\xi\), defined above, only disagree in the angle variable. 
This, in particular, means that the pull-backs of the euclidean metric by the two transformations define equivalent distances.  This allows us to define the H\"older spaces using either transformation.

Further, in either case, the function \(\vert \fz \vert^2\),
 which is the  K\"ahler potential of the flat metric on \(\bbC\),
is sent to the function \(\vert \fz \vert^{2 \over p}\). 
We have defined the spaces \(C^{k,\al}_\ta\) 
in such a way that functions such as \(\vert \fz \vert^{2 \over p}\) 
will belong to them.

From this picture it should be clear that there is a more intrinsic way of defining the H\"older spaces \(C^{\al, \be}\) which is the content of the following definition. 
We shall indeed make use of this equivalence later.

\begin{definition}
	For a given function \(f\), define the semi-norm \([f]^\ta_{\al}\) as follows:
	\[
		[f]^\ta_{\al} := \sup_{x, y} {|f(x) - f(y)| \over d_\ta(x,y)^\al}
	\]
	Also, define the \(C^{\al}_\ta\) norm of the function \(f\) to be 
	\(\Vert f \Vert_{\al,\ta} = \Vert f \Vert_0 + [f]^\ta_{\al}\),
	 and let \(C^{\al}_\ta\) designate the space of all functions having finite \(\Vert . \Vert_{\al,\ta}\) bound.
\end{definition}

Using the \(C^{\al}_\ta\) spaces for functions, we can now define \(C^{\al}_\ta\) forms and thereby define the space \(C^{2,\al}_\ta\).

\begin{definition}
	Let \(\si\) be a \((1,0)\)-form. We say define:
	\[
		\Vert \si \Vert_{\al,\be} := \sup 
	\]
\end{definition}

Following Donaldson's definition of edge-cone H\"older spaces we recall the following definitions.

\begin{definition}
We say that a function \(f\) is \(C^\al(\calS_\ta)\) provided that \(f \) is \(\al\)-H\"older continuous on the the sector \(\arg \zeta \in [0,2 \pi \ta)\) with the two rays \(\ta =0\) and \(\ta = 2 \pi \ta\) identified outside of the origin.
The spaces \(C^{k,\al}\) are defined similarly.
\end{definition}

\section{Setting up the equation}

\label{mise-en-equation}

In this section, we shall see how to reduce the proof of Theorem \ref{thm-1-1} can be reduced to Theorem \ref{thm-1-2}.
To this end, we study the behaviour of various components of the complex Monge-Amp\`ere equation written with respect to 
the reference metrics \eqref{cone-reference} and \eqref{good-metric-eq-0}.
Here we will assume some of the results in \textsection \ref{good-metric-sec}.

In the case of  smooth background K\"ahler metric \(\om_0\) on a manifold which satisfies \(\m \om_0 \in  c_1(M) \), when one tries to find a K\"ahler-Einstein metric with constant \(\m \in \{-1, 0, 1\}\), the function \(f\)  on the right hand side of the complex Monge-Amp\`ere equation  \((\om_0 + dd^c \phi)^n = e^{f - \m\phi} \om_0^n \) can be obtained by solving the following equation for \(f\): 
\[
	\ro(\om_0) - \m \om_0 = dd^c f
\]

However, as we shall see in \textsection \ref{good-metric-sec} the Ricci-form of the family of metrics \(\om_\eps\) and \(\tilde{\om}_\eps\) may not be bounded in general, the former from below and latter from above. 
This lack of lower bound on the curvature of  \(\om_0 + dd^c |s|_h^{2 \ta}\) has  already been well-known as a difficulty in the theory of edge-cone K\"ahler spaces as certain estimates depend on the lower bound of the curvature of the reference metric.
But This, in particular, implies that we cannot assume a laplacian bound on \(f\).

To make this idea more quantitative, we recall that by a \emph{edge-cone K\"ahler-Einstein} metric, \(\om\), of cone angle \(\ta\) we mean one that satisfies \(\ro(\om) = \m \om +2\pi (1 - \ta) [D]\).
Similar to the smooth case, we are lead to the following equation for \(f\)
\begin{equation}
\label{ricci-potential}
	dd^c f = \ro(\om_{\fK}) - \m \om_{\fK} - 2 \pi(1 - \ta) [D]
\end{equation}
wherein \(\om_{\fK}\) is an arbitrary edge-cone background metric.

First, as we shall see in \textsection 4,  the curvature of the reference metric 
\(\tilde{\omega}_\ta \)
 becomes unbounded from above at the rate of \(|s|^{2\ta - 4}\), more precisely, we have:
 \[
	 \tilde{R}(.,.,.,.) \geq C g_\ta \boxtimes g_\ta |s|^{2 \ta' -4\ta}
 \]
wherein \(\boxtimes\) denotes the Kalkurni-Nomizu product of tensors. 
As a result, the Ricci form of the metric also behaves asymptotically as:

\[
	\varrho(\tilde{\om}_\ta) \geq C g_{\ta} |s|^{2\ta' -4\ta}
\]

Therefore, in terms of the asymptotic behaviour close to the divisor in 
\ref{ricci-potential} the term \(\varrho(\tilde{\om}_\ta)  \) dominates the metric. 
As a result, the form
\(dd^c f \) is  bounded from below by a multiple of the K\"ahler form \(\tilde{\om}_\ta \), and the laplacian 
\(\De_\ta f\) is bounded from below.

Similarly, assume that instead of a K\"ahler-Einstein metric we seek to realise a prescribed Ricci form \(\tilde{\varrho} \).
The potential \(f\) will then haveto satisfy the following:
\[
	dd^c f = \varrho(\tilde{\om}_\ta) - \tilde{\varrho}
\]
Again, we see that so long as 
\((\tr_{\om_\ta} \tilde{\ro})^+ = \calO(|s|^{-2\ta })\) we can find a \(\ta' \in (\ta,1) \) so that 
the Ricci form \(\tilde{\varrho} (\tilde{\om}_\ta) \) 
will have blow up at a higher rate and will, thereby, 
dominate the behaviour and will make 
\(dd^c f \) a current bounded from below.
In particular, this means that \(\De_\ta f \) will be bounded from below which is what we need for the laplacian estimate.

Having made these observations, one can follow the usual way of reducing the statement of 
Theorem 
\ref{thm-1-1}
to that of Theorem \ref{thm-1-2}.

\section{Choosing a good metric and approximation by  smooth family}
\label{good-metric-sec}
In deriving the laplacian estimates, it is evident in the calculations of the Pogorelov-Aubin-Yau approach that the estimate depends on the lower bound of the bisectional curvature. 
However, once we add a potential of the form \(\vert s \vert^{2 \ta}_h\) to a smooth background metric
in the vicinity of the zero locus of \(s\) 
the curvature might become unbounded from below , for \(\ta > {1 \over 2}\).
Outside of the divisor, the curvature of such metric is always bounded from above, however, as it has been shown in the appendix of 
\cite{je-ma-ru}.
\begin{comment}
  Brendle has extended the solution of the Calabi problem to the case when the curvature is bounded from below, in particular, when \(\be \leq {1 \over 2}\)
\cite{br}.
In \cite{je-ma-ru}, a different approach has been adopted based on the Chern-Lu inequality which allows to work with the case where curvature might be bounded from above and not below.
In 
\cite{gu-pa} a different approach has been proposed for dealing with the lack of lower bounds on the curvature via introducing an auxiliary function in the calculations of the laplacian estimates. 
This allows using a modification of the Pogorelov-Aubin-Yau approach rather than the Chern-Lu inequality. 
\end{comment}

Here, we shall show how one could perturb the metric in the same -edge-  cohomology class so that the curvature of the metric will become bounded from below.
Indeed, the perturbed  meric and the metric 
\(\om_\ta\) 
are close in a suitable H\"older norm.
Hence, bounding the laplacian with respect to one of them will suffice for deriving {\it a priori} bounds on the complex hessian.

The smallness of the parameter \(\fc\) guarantees that the (1,1)-form thus obtained is positive definite and therefore a K\"ahler metric.
Obviously,  we can always scale the metric \(h\) defined on the line bundle so that the expression \(\tilde{\om}=\om_0 + dd^c |s|_h^{2 \ta} - dd^c |s|_h^{2 \ta'}\) is positive definite. 
Indeed we shall assume this from now on and will drop the coefficients \(\fa\) and \(\fc\).
	It is easy to observe in the proof that the correction \(|s|_h^{2 \ta'}\) to the potential does not change the curvature properties of the metric so much at the points away from the divisor. 

Since we are going to construct the solution as the limit of a sequence of smooth solutions, each of which is obtained by solving with respect to a smooth reference metric, we need the following lemma which allows proving uniform estimate for a sequence of smooth approximations of \(\tilde{\om}_\eps\).

\begin{lemma} 
\label{good-metric-lemma}
The family \(\{ \om_\eps \}_\eps\) defined as:

\begin{equation}
\label{good-metric-eq}
	\tilde{\om}_\eps = \om_0 + dd^c \big( (|s|_h^2 + \eps )^ \ta - (|s|_h^2 + \eps)^ {\ta'} \big)
\end{equation}
is a smooth family approximating \(\tilde{\om}\). 
Further, elements of \(\tilde{\om}_\eps\) have a uniform (in the parameter \(\eps\)) lower bound on their curvature

\end{lemma}

Indeed, what we shall prove is that in curvature of these metrics satisfy
\(R^\eps(v, \bar{v}, w, \bar{w}) \geq C \vert v \vert_\ta  \vert w \vert_\ta\) for some constant \(C\).
The following proof essentially shows both Proposition \ref{good-metric-pro} and Lemma \ref{good-metric-lemma}. 
Indeed the lemma is just a small but useful observation about the content of  Proposition \ref{good-metric-pro}. 
The rest of this section is dedicated to the proof of the proposition.

\subsection{Proof of Proposition \ref{good-metric-pro}}

Let us notice that a function of the form \(\eta = \vert s \vert_h^{2 \ta'}\), where \(\ta' \in ( \ta', 1)\), belongs to \(C^{2,\al}_\ta\) for some exponent \(\al\) in the sense defined before.
In particular, \(\eta\) is a valid K\"ahler potential to be added to the metric.
Let us first  prove that  
\(\tilde{\om}\) is indeed a metric equivalent to \(\om_\ta\).
It will suffice to show that the complex hessian of the correction potential, \(dd^c |s|^{2 \ta'}\), is bounded, as a (1,1)-form, when measured with respect to the metric \(\om_\ta\).
Just as it is done already in \eqref{tilde-omega-good-coord}, by calculating the complex hessian of the potential \(|s|_h^{2 \ta'}\) in the special coordinates we see that the expression 
\[ \Vert \tilde{\om} - \om_\ta \Vert_{\om_\ta}^2 = g^{\ka \bla}_\ta g^{\m \bn}_\ta \big( |s|_h^{2 \ta'} \big)_{,\m \bla} \big( |s|_h^{2 \ta'} \big)_{,\ka \bn} \]
consists of terms of the form \(g^{\ka \bla}_\ta g^{\m \bn}_\ta M_{,\ka \bn} M_{,\m \bla} |\fz^1|^{4 \ta'}   \), which are finite, and when all indices are equal to 1, other terms which are  dominated by \(g^{1 \bar{1}}_\ta g^{1 \bar{1}}_\ta M^2  |\fz^1|^{4 \ta' - 4} \).
Upon noticing that \(g^{1 \bar{1}} = \calO( |\fz_1|^{2 - 2\ta})\), and that \(\ta' > \ta\), one concludes that  \( \Vert \tilde{\om} \Vert_{\om_\ta} = \calO(|\fz^1|^{2( \ta' - \ta)}  ) \) which is finite and, indeed, tends to zero as the points approach the divisor.

In order to study the curvature tensor we shall use the following lemma which will simplify our calculations. 
This observation seems to have been first stated in 
\cite{ti-ya}, in the proof of Lemma 4.3 in \cite{ti-ya}, and also used in the proof of the existence of an upper bound on the curvature of edge-cone metrics in 
\cite{je-ma-ru}.
The version we state below as well as the proof may be found in \cite{ca-gu-pa}, where it appears as Lemma 4.1.

\begin{lemma} 

\label{good-coord}
\label{lemma-ti-ya} 

Assume that for \(j \in \{1, ..., n_0\}\) the \(D_j\)'s are irreducible divisors with at most normal crossing with associated  line bundles \(L_j\) with hermitian metrics \(h_j\) and defining sections of the corresponding line bundles \(s_j\).
Let \(p \in \cap D_j \). 
Then, there is a neighbourhood of \(p_0\) in which for any point \(p\) there is a choice of local coordinates for the manifold \(M\) and trivialisations \(\tta_j\) for the line bundles \(L_j\) on an open set \(U\) so that 

\begin{enumerate}
	\item the hypersurfaces locally correspond to flat hyperplanes: 
	\(U \cap D_j = [z_j =0]\), 
	\item if the herimitian metric \(h_j\) is represented by \(e^{-\phi}\) in the trivialisation \(\tta_j\), then at the point \(p\) we have 
	\[
		\phi_j(0) = 0,  \quad d \phi_j(p) = 0,  \quad 
		{\del \phi \over \del \fz^\al \del \fz^\be}(p)= 0.
	\]
	
\end{enumerate}

	Further, all higher derivatives of \(\phi\) are  bounded uniformly when the point varies on a compact subset of \(U\).

\end{lemma}

\begin{remark}
	It is probably worth mentioning that the above lemma will be used to estimate the rate of blow-up of quantities in terms of \(\fz_1\)
	However, in order for such an estimate to make sense one has to also notice that although the coordinates chosen do  depend on the point \(p\), all the coordinates chosen, in particular \(\fz_1\), are uniformly equivalent as the point \(p\) varies on a compact set.
	In particular, it is well-defined to speak of the rate of blow-up or the rate of vanishing in terms of powers of \(\fz^1\).
\end{remark}

By bounds on the curvature we mean \(R(v, \bar{v}, w, \bar{w}) = R_{\al \bbe \ga \bde} v^\al \bar{v}^{\be} w^\ga \bar{w}^\de\) when \(v= \del_\al v^\al\) and \(w = \del_\al w^\al\) are of unit norm with respect to the edge-cone metric \(\om_{\fK}\).
Since in our adopted coordinate system the metric satisfies \(g_{1 \bar{1}} \approx |\fz_1|^{2\ta - 2}\),
we have that \(\vert v_1 \vert , \vert w_1 \vert \leq C |\fz_1|^{1 - \ta \over 2}\).
Therefore, in studying the terms appearing in the curvature tensor, we shall consider only the once that persist as -potentially- infinite terms after multiplying   \( |\fz^1|^{2 - 2\ta}\).

We have to prove that the   family of metrics \(
\tilde{\om}_\eps = \om_0 + dd^c  \big ( (\eps + |s|_h^2)^ {\ta} + {1 \over N} (\eps + |s|_h^2)^{\ta'} \big )\) has a uniform lower bound on the curvature tensor.
 The coefficient \({1 \over N}\) is added to make sure that \(\tilde{\om}_\eps\) stays positive definite. It is easy to see that the metric \(h\) on the line bundle \(L\) can always be scaled so that the positivity condition holds for the following family of metrics defined in \eqref{good-metric-eq}.
Therefore, without loss of generality, we shall assume from now on that \(N=1\).
One can find the curvature tensor \(R^\eps_{\al \bbe \ga \bde}\) by differentiating the metric. 
However, one may note that since the components of the curvature tensor, written in coordinates, are combinations of various powers of \( \eps + |s|_h^2 \) and \(|s|_h\), it will suffice to show the claim when the parameter \(\eps \) is zero and that will prove the proposition for the entire range of the parameter \(\eps\).

In the coordinate system constructed in Lemma \ref{lemma-ti-ya} we can write \(|s|_h^{2 \ta} = a^\ta |\fz|^{2 \ta}\), and also \(|s|_h^{2 \ta'} = a^\ga |\fz|^{2 \ta'}\). Let us keep make the following substitution in order to keep the notation simpler: \(K:= a^\ta\), \(M := a^{\ta'}\).
Evidently, in the special coordinates of Lemms \ref{good-coord}
we have at the point \(p\) that \(M(p)=K(p)=1\), \(dK(p)=dM(p)=0, K,_{\al \be} = M,_{\al \be} = 0\), for \( \al, \be = 1 ... n\).
Since we have taken \(\eps = 0\), we have:
\[
	\tilde{g}_{\al \bbe} = 
	g^0_{\al \bbe} + \big( K |\fz|^{2 \ta} - M |\fz|^{2 \ta'} \big)_{,\al \bbe}
\]
By differentiating directly we obtain:
\begin{eqnarray}
\label{tilde-omega}
	\tilde{g}_{\al \bbe} &=& g^0_{\al \bbe} 
	+ K_{,\al \bbe} |\fz|^{2 \ta}
	+ \ta \de_{1\be} K_{,\al} |\fz|^{2 \ta -2} \fz 
	+ \ta \de_{1 \al} K_{,\bbe} |\fz|^{2 \ta - 2} \bfz 
	+ \ta^2 \de_{1 \al}  \de_{1 \be} K |\fz|^{2 \ta -2} \nonumber \\
	&-&
	 M_{,\al \bbe} |\fz|^{2 \ta'}
	- \ta' \de_{1\be} M_{,\al} |\fz|^{2 \ta' -2} \fz 
	- \ta' \de_{1 \al} M_{,\bbe} |\fz|^{2 \ta' - 2} \bfz 
	- \ta'^2 \de_{1 \al}  \de_{1 \be} M |\fz|^{2 \ta' -2}  
\end{eqnarray}
Which, in the coordinates of Lemma \ref{lemma-ti-ya} simplifies to the following
\begin{eqnarray}
\label{tilde-omega-good-coord}
	\tilde{g}_{\al \bbe} &=& g^0_{\al \bbe} 
	+ K_{,\al \bbe} |\fz|^{2 \ta}
	+ \ta^2 \de_{1 \al}  \de_{1 \be} K |\fz|^{2 \ta -2}
	-
	 M_{,\al \bbe} |\fz|^{2 \ta'}
	- \ta'^2 \de_{1 \al}  \de_{1 \be} M |\fz|^{2 \ta' -2} \nonumber 
\end{eqnarray}
We now have for the first derivatives of the metric that
\begin{eqnarray}
	\tilde{g}_{\al \bbe,\ga} &=& 
	g^0_{\al \bbe, \ga} + K_{,\al \bbe \ga} |\fz|^{2 \ta} 
	+
	\ta \de_{1 \be} K_{,\al \ga} |\fz|^{2 \ta -2} \fz 
	+
	\ta ( \de_{1 \al} K_{,\bbe \ga} + \de_{1 \ga} K_{,\al \bbe} ) |\fz|^{2\ta - 2} \bfz \nonumber \\
	&+&
	\ta^2 (\de_{1 \be} \de_{1 \ga} K_{,\al} + \de_{1 \al} \de_{1 \be} K_{,\ga}
	+ \de_{1 \al} \de_{1 \ga} K_{,\bbe} ) |\fz|^{2 \ta - 2}
	+ \ta^2 ( \ta - 1) \de_{1 \al} \de_{1 \be} \de_{1 \ga} K  |\fz|^{2 \ta - 4} \bfz \nonumber \\
	&-& M_{,\al \bbe \ga} |\fz|^{2 \ta'} 
	-
	\ta' \de_{1 \be} M_{,\al \ga} |\fz|^{2 \ta' -2} \fz 
	-
	\ta' ( \de_{1 \al} M_{,\bbe \ga} + \de_{1 \ga} M_{,\al \bbe} ) |\fz|^{2\ta' - 2} \bfz \nonumber \\
	&-&
	\ta'^2 (\de_{1 \be} \de_{1 \ga} M_{,\al} + \de_{1 \al} \de_{1 \be} M_{,\ga}
	- \de_{1 \al} \de_{1 \ga} M_{,\bbe} ) |\fz|^{2 \ta' - 2}
	- \ta'^2 ( \ta' - 1) \de_{1 \al} \de_{1 \be} \de_{1 \ga} M |\fz|^{2 \ta' - 4} \bfz	
\end{eqnarray}
and in the coordinates of Lemma \ref{lemma-ti-ya} this simplifies to:
\begin{eqnarray}
	\tilde{g}_{\al \bbe,\ga} &=& 
	g^0_{\al \bbe, \ga} + K_{,\al \bbe \ga} |\fz|^{2 \ta} 
	+
	\ta ( \de_{1 \al} K_{,\bbe \ga} + \de_{1 \ga} K_{,\al \bbe} ) |\fz|^{2\ta - 2} \bfz 
	+ \ta^2 ( \ta - 1) \de_{1 \al} \de_{1 \be} \de_{1 \ga} K  |\fz|^{2 \ta - 4} \bfz \nonumber \\
	&-& M_{,\al \bbe \ga} |\fz|^{2 \ta'} 
	-
	\ta' ( \de_{1 \al} M_{,\bbe \ga} + \de_{1 \ga} M_{,\al \bbe} ) |\fz|^{2\ta' - 2} \bfz 
	- \ta'^2 ( \ta' - 1) \de_{1 \al} \de_{1 \be} \de_{1 \ga} M  |\fz|^{2 \ta' - 4} \bfz
\end{eqnarray}

And similarly, the expression for \(g_{\al \bbe, \bde}\) in the coordinates of Lemma \ref{lemma-ti-ya} is:
\begin{eqnarray}
	\tilde{g}_{\al \bbe,\bde} &=& 
	g^0_{\al \bbe, \bde} + K_{,\al \bbe \bde} |\fz|^{2 \ta} 
	+
	\ta (\de_{1 \be} K_{,\al \bde} 
	+
	  \de_{1 \al} K_{,\bbe \bde}) |\fz|^{2 \ta -2} \fz + \de_{1 \de} K_{,\al \bbe}  |\fz|^{2\ta - 2} \bfz \nonumber \\
	&+&
	\ta^2 (\de_{1 \be} \de_{1 \de} K_{,\al} + \de_{1 \al} \de_{1 \be} K_{,\ga}
	+ \de_{1 \al} \de_{1 \de} K_{,\bbe} ) |\fz|^{2 \ta - 2}
	+ \ta^2 ( \ta - 1) \de_{1 \al} \de_{1 \be} \de_{1 \de} K  |\fz|^{2 \ta - 4} \fz \nonumber \\
	&-& M_{,\al \bbe \bde} |\fz|^{2 \ta'} 
	-
	\ta' (\de_{1 \be} M_{,\al \bde} 
	-
	 \de_{1 \al} M_{,\bbe \bde}) |\fz|^{2 \ta' -2} \fz +
	 \de_{1 \de} M_{,\al \bbe}  |\fz|^{2\ta' - 2} \bfz \nonumber \\
	&-&
	\ta^2 (\de_{1 \be} \de_{1 \de} M_{,\al} + \de_{1 \al} \de_{1 \be} M_{,\ga}
	- \de_{1 \al} \de_{1 \de} M_{,\bbe} ) |\fz|^{2 \ta - 2}
	- \ta'^2 ( \ta' - 1) \de_{1 \al} \de_{1 \be} \de_{1 \de} M  |\fz|^{2 \ta - 4} \fz 
\end{eqnarray}

Also, for the second derivative terms we have in the special coordinates that:
\begin{eqnarray}
	\tilde{g}_{\al \bbe, \ga \bde} \text{  } &=& 
	g^0_{\al \bbe, \ga \bde} + K_{,\al \bbe \ga \bde} |\fz^1|^{2 \ta} 
	+ 
	\ta \big( \de_{1 \be} K_{,\al \ga \bde} + \de_{1 \de} K_{,\al \bbe \ga}  \big) |\fz^1|^{2 \ta - 2} \fz^1	\nonumber \\
	&+&
	\ta \big( \de_{1 \ga} K_{,\al \bbe \bde} + \de_{1 \al} K_{,\bbe \ga \bde} \big) |\fz^1|^{2\ta -2} \bfz_1		 \nonumber \\
	&+& \ta^2    
	\big ( \de_{1 \al} \de_{1\de} K_{,\ga \bde} 
	+ \de_{1 \ga} \de_{1 \de} K_{,\al \bbe}
	+ \de_{1 \al} \de_{1 \be} K_{,\ga \bde}
	+ \de_{1 \be} \de_{1 \ga} K_{,\al \bde}
		\big) |\fz^1|^{2\ta -2}  \nonumber \\
	&+& \ta (1 - \ta) \de_{1 \be} \de_{1 \de}  K_{,\al \ga} |\fz_1|^{2\ta - 4} \fz_1^2
	+ \ta ( 1- \ta) \de_{1 \al} \de_{1 \ga} K_{,\bbe \bde}  |\fz_1|^{2\ta - 4} \bfz_1^2
	\big ) |\fz_1|^{2 \ta -2}
	\nonumber \\
	&+& \ta^2 ( 1 - \ta) \de_{1 \al} \de_{1 \ga} 
	\big(
		\de_{1 \de} K_{,\bbe}
		+
		\de_{1 \be} K_{,\bde}
	\big) |\fz^1|^{2 \ta -4} \bfz_1
	+ 
	\ta^2 (1 -\ta)  \de_{1 \be} \de_{1 \de} 
	\big(
		\de_{1 \ga} K_{,\al} + \de_{1 \al} K_{,\ga}
	\big) |\fz^1|^{2 \ta - 4} \fz^1
	\nonumber \\
	&+& \ta^2 ( 1 - \ta)^2 \de_{1 \al} \de_{1 \be} \de_{1 \ga} \de_{1 \de} K |\fz^1|^{2 \ta - 4}  			\nonumber \\
	&-&   % new quantities
	 M_{,\al \bbe \ga \bde} |\fz^1|^{2 \ta'} 
	- 
	\ta' \big( \de_{1 \be} M_{,\al \ga \bde} + \de_{1 \de} K_{,\al \bbe \ga}  \big) |\fz^1|^{2 \ta' - 2} \fz^1		
	-
	\ta' \big( \de_{1 \ga} M_{,\al \bbe \bde} + \de_{1 \al} M_{,\bbe \ga \bde} \big) |\fz^1|^{2\ta' -2} \bfz_1		 \nonumber \\
	&-& \ta'^2    
	\big ( \de_{1 \al} \de_{1\de} M_{,\ga \bde} 
	- \de_{1 \ga} \de_{1 \de} M_{,\al \bbe}
	- \de_{1 \al} \de_{1 \be} M_{,\ga \bde}
	- \de_{1 \be} \de_{1 \ga} M_{,\al \bde}
	\big) |\fz^1|^{2\ta' -2} \nonumber \\
	&-& \ta' (1 - \ta') \de_{1 \be} \de_{1 \de}  M_{,\al \ga} |\fz^1|^{2\ta' - 4} \fz_1^2
	- \ta' ( 1- \ta') \de_{1 \al} \de_{1 \ga} M_{,\bbe \bde}  |\fz^1|^{2\ta' - 4} \bfz_1^2
	\big ) |\fz^1|^{2 \ta' -2}
	\nonumber \\
	&-& \ta'^2 ( 1 - \ta) \de_{1 \al} \de_{1 \ga} 
	\big(
		\de_{1 \de} M_{,\bbe}
		+
		\de_{1 \be} M_{,\bde}
	\big) |\fz^1|^{2 \ta' -4} \bfz_1 \nonumber \\
	&-& 
	\ta'^2 (1 -\ta')  \de_{1 \be} \de_{1 \de} 
	\big(
		\de_{1 \ga} M_{,\al} + \de_{1 \al} M_{,\ga}
	\big) |\fz^1|^{2 \ta' - 4} \fz^1
	- \ta'^2 ( 1 - \ta')^2 \de_{1 \al} \de_{1 \be} \de_{1 \ga} \de_{1 \de} M |\fz^1|^{2 \ta' - 4}
\end{eqnarray}

Which in the special coordinate system adopted at a given point becomes:
\begin{eqnarray} \label{g-second-deriv}
	\tilde{g}_{\al \bbe, \ga \bde} &=& 
	g^0_{\al \bbe, \ga \bde} + K_{,\al \bbe \ga \bde} |\fz^1|^{2 \ta} 
	+ 
	\ta \big( \de_{1 \be} K_{,\al \ga \bde} + \de_{1 \de} K_{,\al \bbe \ga}  \big) |\fz^1|^{2 \ta - 2} \fz^1			 \nonumber \\
	&+&
	\ta \big( \de_{1 \ga} K_{,\al \bbe \bde} + \de_{1 \al} K_{,\bbe \ga \bde} \big) |\fz^1|^{2\ta -2} \bfz_1		 \nonumber \\
	&+& \ta^2    
	\big ( \de_{1 \al} \de_{1\de} K_{,\ga \bde} 
	+ \de_{1 \ga} \de_{1 \de} K_{,\al \bbe}
	+ \de_{1 \al} \de_{1 \be} K_{,\ga \bde}
	+ \de_{1 \be} \de_{1 \ga} K_{,\al \bde}
		\big) |\fz^1|^{2\ta -2}  \nonumber \\
	&+& \ta^2 ( 1 - \ta)^2 \de_{1 \al} \de_{1 \be} \de_{1 \ga} \de_{1 \de} K |\fz^1|^{2 \ta - 4} 			\nonumber \\
	&-&   % new quantities
	 M_{,\al \bbe \ga \bde} |\fz^1|^{2 \ta'} 
	- 
	\ta \big( \de_{1 \be} M_{,\al \ga \bde} + \de_{1 \de} K_{,\al \bbe \ga}  \big) |\fz^1|^{2 \ta' - 2} \fz^1			 \nonumber \\
	&-&
	\ta' \big( \de_{1 \ga} M_{,\al \bbe \bde} + \de_{1 \al} M_{,\bbe \ga \bde} \big) |\fz^1|^{2\ta' -2} \bfz_1		 \nonumber \\
	&-& \ta'^2    
	\big ( \de_{1 \al} \de_{1\de} M_{,\ga \bde} 
	- \de_{1 \ga} \de_{1 \de} M_{,\al \bbe}
	+\de_{1 \al} \de_{1 \be} M_{,\ga \bde}
	+ \de_{1 \be} \de_{1 \ga} M_{,\al \bde} 
		\big) |\fz^1|^{2\ta' -2} \nonumber \\
	&-& \ta'^2 ( 1 - \ta')^2 \de_{1 \al} \de_{1 \be} \de_{1 \ga} \de_{1 \de} M |\fz^1|^{2 \ta' - 4}
\end{eqnarray}

In order to study the behaviour of the terms with the inverse of the metric  close to the divisor in its coordinate representation we shall need the following observation.
\begin{lemma} \label{good-metric-inverse}
	For the inverse matrix \(g^{\m \bn}\) we have:
	\begin{itemize}
		\item \(g^{1 \bar{1}} = {1 \over \ta^2} |\fz^1|^{2(1 - \ta)} \left(
		 1 + {\ta'^2 \over \ta^2}  |\fz^1|^{2(\ta' - \ta)} + \calO \big(   |\fz^1|^{2(1 - \ta)}\big) \right)  \),
		\item \(g^{\m \bar{1}} = \calO \big( |\fz^1|^{2(1 - \ta)} \big)\), for \(\m \neq 1\).	
	\end{itemize}
\end{lemma}

\begin{proof}
			We shall derive the aymptotic behaviour of the elements of \(g^{i \bj}\) by studying the components \(g^{i\bj}\) as quotients \({\det g^*_{i j} \over \det g} \), where \(g^*_{ij}\) is the minor obtained by removing the \(i\)-th row and \(j\)-th column.
			Since in the coordinate systems introduced in Lemma \ref{good-coord} the only unbounded component of the metric tensor will be \(g^\eps_{1 \bar{1}} = 2 (\ta - 1) \calO( |\fz^1|^{2(\ta -1)}) + \text{bounded terms}\), we see by expanding the determinant in the first row that:
			\[
				\det g = \big[ 
				\ta^2|\fz^1|^{2(\ta -1)}
				-
				\ta'^2 |\fz^1|^{2(\ta' -1)}
				\big] \det g^*_{11} + \text{bounded terms}
			\]
			and that \(\det g^*_{11}\) is bounded.
			Therefore, 
			\({ \det g^*_{11} \over \det g} = 
			{\det g^*_{11}  \over [ \ta^2 |\fz^1|^{2(\ta -1)} 
			- 
			\ta'^2 |\fz^1|^{2(\ta' -1)}] \det g^*_{11}  + a_1}
			\)
%			[[Notice that \(a_1\) might have any sign as it is an alternating sum of positive terms, at the best. But the good news might be that we know the sign of \((...)^{\ga - \be}\).]].\\
			where \(a_j\)'s are bounded functions. One can now expand the quotient as 
			\begin{eqnarray}
			\label{g-inverse-expansion}
				g^{1 \bar{1}}={ \det g^*_{11} \over \det g} &=& 
				{1 \over \ta^2} |\fz^1|^{2(1 - \ta)} \Big[
				1 
				+
				{\ta'^2 \over \ta^2}   |\fz^1|^{2(\ta' - \ta)}
				+
				{\ta'^2 \over \ta^2}{a_1 \over \det g_{11}^*} 
				 K|\fz^1|^{2(1 - \ta)}  \nonumber \\
				&+& \calO \big(
				|\fz^1|^{4( \ta' -\ta )} \big)
				\Big]
			\end{eqnarray}
			As it will become clear in the calculations of the curvature tensor, it is indeed the first order expansion of the fraction that is of importance for our purpose. 
			Indeed, when we do not have the extra term \( |\fz_1|^{2\ta}\) in the potential, 
			the term that could produce the negative infinity in the curvature is the term \( {-a_1 \over \det g^*_{11}} |\fz_1|^{2(1 - \ta)}\).
			However, in our case this latter term is dominated by the term \( |\fz_1|^{2(\ta' - \ta)}\).
			
			For components \(g^{i \bar{1}}, i \neq 1\) we can also derive a similar expression, the difference being that we do not need first order expansion, the signs of the terms appearing are not known, and do not play a r\^{o}le in the unboundedness of curvature tensor.		
\end{proof}

Having found the derivatives and the inverse of the metric (as a matrix in local coordinates) we can now turn to finding the curvature. 		
Let us recall that we have the following formula for the curvature of a K\"ahler metric:
\begin{equation}
			\label{curv-form}
			R_{\al \bbe \ga \bde} = -  g_{\al \bbe, \ga \bde} + g^{\m \bn} g_{\al \bn, \ga} g_{\m \bbe, \bde}
\end{equation}

It is not hard to see that the only terms in the second derivative, \(\tilde{g}_{\al \bbe, \ga \bde}\), which need to be considered are the following: 
\(  \de_{1 \al} \de_{1 \be} \de_{1 \ga} \de_{1 \de} \big (   \ta^2 ( 1 - \ta)^2  K |\fz^1|^{2 \ta - 4}  - \ta'^2 ( 1 - \ta')^2 M |\fz^1|^{2 \ta' - 4} \big )\).
This  term is non-zero only when all indices are equal to 1.
In the first order expression \(g^{\m \bn} g_{\al \bn, \ga} g_{\m \bbe, \bde}\) one can first notice that \(g^{1 \bar{1}} = \calO( |\fz^1|^{2 - 2\ta} ) \).
So, the only term   in the product \(g_{\al \bn, \ga} g_{\m \bbe, \bde}\) 
that might stay unbounded after multiplying into the relevant components of \(v\) and \(w\) is the following
\begin{eqnarray}
	\big(  
		 \ta^2 ( \ta - 1) \de_{1 \al} \de_{1 \be} \de_{1 \de} K  |\fz|^{2 \ta - 4} \fz
	&-&
	\ta'^2 ( \ta' - 1) \de_{1 \al} \de_{1 \be} \de_{1 \de} M  |\fz|^{2 \ta' - 4} \fz 
 \big) 
	\big(
	  \ta^2 ( \ta - 1) \de_{1 \al} \de_{1 \be} \de_{1 \ga} K  |\fz|^{2 \ta - 4} \bfz 	
	  \nonumber \\
	  &-& 
	 \ta'^2 ( \ta' - 1) \de_{1 \al} \de_{1 \be} \de_{1 \ga} M  |\fz|^{2 \ta' - 4} \bfz 	  
	\big)
\end{eqnarray}
and that only when \(\al = \be = \ga = \de =1\).
It takes a straightforward verification to see that after multiplying by relevant components of \(v\) and \(w\), the only potentially unbounded terms are in \(R_{1 \bar{1} 1 \bar{1}} \).

Let us study the terms in \(R_{\al \bbe \ga \bde}  v^\al \bar{v}^{\be} w^\ga \bar{w}^\de\) separately. 
First, we take 
\begin{equation} \label{curv-1}
	-g_{\al \bbe, \ga \bde} v^\al \bar{v}^{\be} w^\ga \bar{w}^\de
\end{equation}
By directly inspecting the terms in 
\eqref{g-second-deriv} we see that the only unbounded term in  
\eqref{curv-1} is the highest order term
\[
  \big(  - \ta^2 ( 1 - \ta)^2  K |\fz^1|^{2 \ta - 4}
  + \ta'^2 ( 1 - \ta')^2  M |\fz^1|^{2 \ta' - 4}
    \big) v^1 \bar{v}^1 w^1 \bar{w}^1.
\]

Behaviour of the term 
\(g^{\m \bn} g_{\al \bn, \ga} g_{\m \bbe, \bde} v^\al \bar{v}^{\be} w^\ga \bar{w}^\de\) can similarly be understood as follows: all the terms appearing in the product are bounded except the term

\begin{eqnarray}
	g^{\m \bn} 
	\de_{1 \m} \de_{1 \be} \de_{1 \de}
	\de_{1 \al} \de_{1 \n} \de_{1 \ga}
	\big( 
	 \ta^2 ( \ta - 1) K  |\fz^1|^{2 \ta - 4} \fz^1 
	&-&
	\ta'^2 ( \ta' - 1)  M  |\fz^1|^{2 \ta' - 4} \fz^1
	\big)
	\big(
	  \ta^2 ( \ta - 1)  K  |\fz^1|^{2 \ta - 4} \bfz_1	
	  \nonumber \\
	  &-& 
	 \ta'^2 ( \ta' - 1)  M  |\fz^1|^{2 \ta' - 4} \bfz_1	 
	\big) v^\al \bar{v}^\be w^\ga \bar{w}^\de \nonumber
\end{eqnarray}
which, using \eqref{g-inverse-expansion} can be written as
\begin{eqnarray}
	{1 \over \ta^2}  \Big(
				1 
				&+&
				{\ta'^2 \over \ta^2}  |\fz^1|^{2(\ta' - \ta)}
				+
				\calO (|\fz^1|^{2(1 - \ta)} )
				\Big)
				 \big[
		\ta^4 (1-\ta)^2  |\fz^1|^{2 \ta -4}
		\nonumber \\
		&-&
		2 \ta^2 \ta'^2 (1 -\ta) (1-\ta') |\fz^1|^{4 \ta' - 2 \ta -4}
		+
		\ta'^4 ( 1- \ta')^2  |\fz^1|^{ 2 \ta' - 4}
	\big]
	v^1 \bar{v}^1 w^1 \bar{w}^1  \nonumber \\
	&=&
	\big( \ta^2 (1-\ta)^2 |\fz^1|^{2 \ta -4} 
	-
	2 \ta'^2 (1-\ta)(1-\ta') |\fz^1|^{2 \ta' - 4}
	+
	{\ta'^4 \over \ta^2} ( 1 - \ta'^2) |\fz^1|^{4\ta' - 2 \ta - 4}    \nonumber \\
	&+&
	\ta'^2 ( 1- \ta)^2 |\fz^1|^{2\ta'  - 4}
	- 2 {\ta'^4 \over \ta^2} ( 1 - \ta) (1 - \ta') |\fz^1|^{6 \ta' - 4\ta - 4}
	+{\ta'^6 \over \ta^4}(1 - \ta')^2 |\fz^1|^{4\ta' - 2 \ta -4}
	 \nonumber \\
	 &+&
	 \calO(|\fz^1|^{-2}) \big)
	v^1 \bar{v}^1 w^1 \bar{w}^1 \nonumber
\end{eqnarray}

Having obtained these expressions for the individual terms appearing in the components of the curvature tensor, we can now put them together to obtain the following expression for two unit vectors \(v, w\)
\begin{eqnarray}
	R_{1 \bar{1} 1 \bar{1}} v^1 \bar{v}^1 w^1 \bar{w}^1 &=& 
	\calO (1) + \Big[
	- 
	\ta^2 ( 1 - \ta)^2  |\fz^1|^{2 \ta - 4} 
	+
	 \ta'^2 ( 1 - \ta')^2   |\fz^1|^{2 \ta' - 4} 
	\nonumber \\
	&+&	\ta^2 (1-\ta)^2 |\fz^1|^{2 \ta -4} 
	-
	2 \ta'^2 (1-\ta)(1-\ta') |\fz^1|^{2 \ta' - 4}
	+
	{\ta'^4 \over \ta^2} ( 1 - \ta'^2) |\fz^1|^{4\ta' - 2 \ta - 4}    \nonumber \\
	&+&
	\ta'^2 ( 1- \ta)^2 |\fz^1|^{2\ta'  - 4}
	- 2 {\ta'^4 \over \ta^2} ( 1 - \ta) (1 - \ta') |\fz^1|^{6 \ta' - 4\ta - 4}
	+{\ta'^6 \over \ta^4}(1 - \ta')^2 |\fz^1|^{4\ta' - 2 \ta -4}
	 \nonumber \\
	 &+&
	 \calO(|\fz^1|^{-2}) 
	 \Big] v^1 \bar{v}^1 w^1 \bar{w}^1 \nonumber	
\end{eqnarray}

As one may observe, starting from the lowest power, the terms with \(|\fz^1|^{2 \ta -4}\), cancel, and the next lowest power, \(|\fz^1|^{2 \ta' - 4}\), has a positive coefficient: \(\ta'^2(1 - \ta')^2 + \ta'^2(1 - \ta)^2 - 2 \ta'^2 (1- \ta) (1 - \ta') = \ta'^2 (\ta' - \ta)^2 > 0 \). 
This means we can disregard the terms with larger exponents altogether
and the behaviour of the curvature is dominated by the positive term \( (\ta'- \ta)^2 |\fz^1|^{2 \ta' - 4}\).
In particular, this means as \(|\fz^1| \to 0\), the \(	R_{1 \bar{1} 1 \bar{1}} v^1 \bar{v}^2 w^1 \bar{w}^1\) blows up in the positive direction and is bounded from below.

Finally, it should be noted that the components of the curvature tensor other than the \(1 \bar{1} 1 \bar{1}\) component are bounded when multiplied by the corresponding  elements of \(g^{\al \bbe}\).
\qed \\

\begin{remark}
 	Althoug we have not detalied this here, but one can repeat similar calculations for a metric of the form 
 	\(\om_0 + dd^c \big( |s|_h^{2 \ta} \pm \fc |s|_h^{2\ta'} \big) \), \(\fc>0\), and observe that for such metric the curvature indeed blows up in the negative or the positive direction -depending on the sign before \(\fc\)- at the rate of \(|s|_h^{2\ta' - 4}\).
 	In particular, this gives an example of a smooth edge-cone metric whose curvature becomes unbounded from either below or above close to the divisor at a rate faster than \(|s|_h^{-2}\).
This also means one can construct such edge-cone metrics even when \(\ta < {1 \over 2}\).
We find it worthwhile to emphasise that this phenomenon is not exclusive to higher dimensions.  In an identical fashion it is indeed possible to construct cone metrics on \(\bbC\) whose curvatures are unbounded from below or above.
\end{remark}

\section{Proof of the main results-Solving the equation}

Having made the observations in \textsection 3 and 4, we can now prove Theorem \ref{thm-1-2}.

\begin{proof}[Proof of Theorem \ref{thm-1-2}]
In order to do so, we shall approximate the equation by a family of equations with smooth components. 
In this section we first estiablish uniform 
 {\it a priori } estimates which will allow taking limit of the family of solutions.

As we have mentioned before, 
the way we solve the equation with edge reference metric is by approximating the edge metric by a family of smooth metrics and deriving estimates independent of the parameter of the sequence.
We also take a family of smooth functions \(\{ f_\eps \}_\eps\) approximating the source term  \(f\).
That is, we solve the following family of equations
\begin{equation}\label{5-1}
	(\tilde{\om}_\eps + dd^c \phi_\eps)^n = \tilde{\om}_\eps^n e^{f_\eps}
\end{equation}

It is not hard to see that the right hand side converges in \(L^p(\om_0)\) for some \(p\) depending on the angle. 
The fundamental theorem of Ko\l odziej \cite{ko} on the stability in \(L^p\) of the Monge-Amp\`ere operator  comes to our aid to guarantee that since the right hand side converges in \(L^p\),
the potentials obtained as solutions converge to the unique H\"older continuous solution.
This also takes care of the \(L^\infty\) estimate automatically.

Just as in the classical case, in order to derive estimates on the complex hessian, we derive an upper bound on the laplacian. This is the content of Theorem \ref{thm-5-1}.

Finally, we need to derive an estimate on the modulus of continuity of the second derivative, namely, the \(C^{2,\tta}_\ta\) estimates {\it \`a la} Evans and Krylov.
This will be done in the following section.
One can then repeat the usual method based on takign a sequence of solutions, \( \{\phi_{\eps_j} \}_j \),
and if necessary pass to a subsequence and by evoking the 
uniform estimates prove that there is a solution as \(\eps \to 0^+\).
This will thus conclude the proof of Theorem 
\ref{thm-1-2}.
\end{proof}

Without mentioning, the functions and metric appearing below are assumed to be the ones corresponding to the \(\eps\)-approximation of the equation.

\begin{theorem}
	\label{thm-5-1}
	Let \(\phi\) be a \(C^3\) solution of Equation \ref{5-1}.
	Then, we have the following {\it a priori} bound:
	\begin{equation}
		\Vert dd^c \phi_\eps \Vert_{\om_\eps} \leq C
	\end{equation}
	where \(C= C \left (\Vert \phi \Vert_{L^\infty}, \inf R^\eps(.,.,.,.), \inf f_\eps, \inf (\De_{\om_\eps} f_\eps)^- \right)\).
\end{theorem}

\begin{proof}
Other than a few observations, the argument is very similar to the classical argument. The reader can consult \textsection \cite{si}.
We shall drop the subscript \(\eps\) in the rest of the proof.
Let us set \(\De\) and \(\De'\) to be the laplacian associated to  \(\tilde{\om}\) and \(\tilde{\om} + dd^c \phi\) respectively.
In the rest of this section, we drop the subscript \(\eps\).
Let us start with the following well-known inequality:
	\begin{equation}
		\De' \log ( n + \De \phi)
		\geq
		{1 \over n + \De \phi} \big[ \De f + \sum_{\al, \be} \big( 
		- R_{\al \bal \be \bbe} + R_{\al \bal \be \bbe} {1 + \phi_{\al \bal} \over 1 + \phi_{\be \bbe}}
		\big)
   \big]
	\end{equation}
	The reader can refer to \textsection 3.2 of \cite{si} for example.	Using the symmetries of the curvature tensor, we now rewrite the expressions containing curvature terms as 
	\begin{eqnarray}
		\sum_{\al, \be} \big( - R_{\al \bal \be \bbe} + R_{\al \bal \be \bbe} {1 + \phi_{\al \bal} \over 1 + \phi_{\be \bbe}} \big)
		&=& 2\sum_{\al < \be} \big(
			{1 + \phi_{\al \bal} \over 1 + \phi_{\be \bbe}}
			+
			{1 + \phi_{\be \bbe} \over 1 + \phi_{\al \bal}}
			-
			2
		\big)R_{\al \bal \be \bbe} \nonumber 	\end{eqnarray}
	In the original way the laplacian estimate was derived, it depended on the upper bound of the scalar curvature and the lower bound of the bi-sectional curvature.
	This observation, which is already used in \cite{pa}, allows dropping the former requirement.
	
	Let \(C\) be a positive constant so that \(R_{\al \bal \be \bbe} u^\al \bar{u}^\al v^\be \bar{v}^\be \geq -C|u|_\om|v|_\om \).
	Then, upon noticing that the terms 
	\({1 + \phi_{\al \bal} \over 1 + \phi_{\be \bbe}} 	+
			{1 + \phi_{\be \bbe} \over 1 + \phi_{\al \bal}} -2 \)
			are all non-negative, we deduce that 
		\begin{eqnarray}
		\De' \log ( n + \De \phi - C_2 \phi )
		&\geq&
		{1 \over n + \De \phi} \big[ \De f - \sum_{\al, \be} C \big( {1 + \phi_{\al \bal} \over 1 + \phi_{\be \bbe}} 	+
			{1 + \phi_{\be \bbe} \over 1 + \phi_{\al \bal}} -2 
		\big) 
   \big]  - C_2n + C_2 \sum_{\al} {1 \over 1 + \phi_{\al \bal}} \nonumber \\
   		&\geq&   		
   		{(\De f)^- \over n + \De \phi} - 2C \sum_\al {1 \over 1 + \phi_{\al \bal}} - 2 Cn^2
   		- C_2n + C_2 \sum_{\al} {1 \over 1 + \phi_{\al \bal}} 
		\end{eqnarray}
		In the expression above, for a function \(v\), \(v^-:= \min \{ v, 0 \}\).
		
	After choosing \(C_2\) to be sufficiently large, it will require a standard application of the maximum principle to conclude the argument. We deduce that the quantity \(\De \phi\) is bounded by a constant \(C_4=C_4((\De f)^-, \inf_{\al, \be} R_{\al \bal \be \bbe}, \Vert \phi \Vert_0)\). 
	All of these quantities are finite by the fact that we have chosen a reference metric whose curvature is bounded from below, along with the observations in Propositions \ref{good-metric-pro}, that guarantee the the laplacian of the Ricci potential is bounded from below.
\end{proof}

Having proved the bounds on the complex hessian of the solution, we can now turn to proving the H\"older continuity of the second derivative using a version of the so-called Evans-Krylov theory to deduce that the solution \(\phi_0\), obtained as the limit of \(\{\phi_\eps \}_\eps \), belongs to the edge H\"older space \(C^{2,\tta}_\ta\) for some exponent \(\tta\). This will be the the subject of the next section.

\subsection{Change of the reference metric redux: Evans-Krylov theory} \label{ev-kr-sec}
	
	The \(C^{2,\tta}_\ta\) estimates have been studied before. In \cite{ch-do-su-3} such an estimate has been derived for K\"ahler-Einstein metrics. 
Also, in \cite{ca-zh} the Evans-Krylov theory has been extended to the edge-cone setting with the restriction of \(\ta \leq {2 \over 3}\), and in \cite{je-ma-ru} for \(\ta \in (0,1)\).
Later, an interesting argument based on approximating the cone angle by rational numbers and using the geometry of rational cone angles was developed in \cite{gu-pa}.
We assume the validity Evans-Krylov theory on \(\bbC_\ta \times \bbC^{n-1}\), which is what has been established in \cite{ca-zh, je-ma-ru, chu, gu-pa}.

So, we first localise the equation as detailed in the next paragraph.
In the backdrop there is indeed a local change of the background metric which allows considering the equation on the flat edge model.
This is necessary since in the edge case, unlike the smooth case, the second derivatives of the reference metric might not be bounded, so taking derivatives of \eqref{eq-1-4}
in its current form might not help with the proof.

In the case of the edge metrics, the equation that the metric satisfies is the following 
\[
	\ro(\om_\phi) = \m \om_\phi +2 \pi (1- \ta) [D].\] 
%Let us take \(u_\ta := |\fz^1|^{2\ta} + \sum_{j>1} |\fz^j|^2\), the potential that generates the standard edge-cone metric on \(\bbC_\ta \times \bbC^{n-1}\).
In the unit ball in \(\bbC_\ta \times \bbC^{n-1}\) the twisted K\"ahler-Einstein equation can be written as
\begin{equation}
\label{cone-local}
	-dd^c \log (dd^c w)^n =  dd^c \log |\fz_1|^{2 - 2 \ta} + \m dd^c w
\end{equation}
Similar to the smooth case, we obtain the following equation for \(w\):
\begin{equation}
\label{cone-local-2}
	\log \det (w_{,\al \bbe}) = \log |\fz_1|^{2\ta - 2} - \m w + \fH
\end{equation}
for some pluri-harmonic function \(\fH\).

Of course since we have no boundary conditions, there are infinitely many choices of a pluri-harmonic function \(\fH\) which in general satisfy no uniformity of any sort. 
Noting the  fact that this equation is satisfied locally by all K\"ahler-Einstein potentials,
it comes as little surprise that with no boundary conditions prescribed and an undetermined source term, \(\fH\), Equation \ref{cone-local-2} has too many degrees of freedom. 
It might, at the face value, seem like by doing so we have lost a great deal of information.
However, when one  has readily obtained a bound on the complex hessian of the solution \(\phi\) on the manifold, it translates to the fact that \(dd^c w\) can be assumed to be bounded. 
Further, since \(\phi\), and hence \(w\), are bounded in \(L^\infty\), along with the fact that the form \(dd^c w\) is bounded, we see from the equation that \(\Vert \fH \Vert_{L^\infty}\) is {\em a priori} bounded.
Then, since \(\fH\) is a pluri-harmonic, and in particular a harmonic function, it is a well-known fact  \(\fH\) is analytic and  that all the derivatives of \(\fH\) are controlled in terms of the oscillation of \(\fH\). 
(One way to prove this fact is by an application of the mean value theorem. For such a proof, see \cite{gi-tr}).
Although, since we take \(\del_{\ka \bla}\) derivatives, and \(\fH_{,\ka \bla} =0\), we only need the bound on the first derivative of \(\fH\).

%This allows means we can now apply the Evans-Krlov theory on the flat space \(\bbC^n\).

Now, in the global equation, \eqref{eq-1-4}, the Ricci potential, \(f\), is bounded, and in our local equation \(\om_\ta\) is equivalent to the standard edge-cone metric, \(\om_{\fK}\). This, along with the fact that  the potential \(w\) is bounded
allows us to immediately see that  the oscillation of \(\fH\) is bounded, which, in turn, gives a bound on all higher derivatives of \(\fH\). We can summarise these observations as the following lemma:

Now one may apply the Evans-Krylov theory on the flat space to \eqref{cone-local-2} and establish the membership of \(w\), and hence \(\phi\), in \(C^{2,\tta}_\ta\) for some \(\tta \in (0,1)\).

In the case of prescribed Ricci form in Theorem \ref{thm-1-2} there is another consideration to be taken into account, namely the regularity of the Ricci potential when the problem is localied. 
It is not hard to see that the Ricci potential stays H\"older continuous after the change of the background metric and also after localising the problem.
Assuming the H\"older continuity of the Ricci potential when the equation is localised with the falt reference metric, we can apply the result in 
\cite{chu}. 
Amongst the Evans-Krylov estimates derived in the edge-cone setting, the result in this reference is the only one which does not require differentiating the equation, and hence, does not require the boundedness of derivatives of the Ricci potential. 
We may summarise this as it is stated below.

\begin{theorem}
	Assume that in Equation \ref{eq-1-4} the function \(\phi\) satisfies:
		\[
			\Vert \phi \Vert_{L^\infty}, \
			\Vert \vert dd^c \phi \vert_{\om_\ta} \Vert_{L^\infty},
			\
			\Vert f \Vert_{C^\al_\ta}
			 \leq C
		\]
		for some constant \(C>0\).
\end{theorem}

We notice that the condition on the H\"older continuity of \(f\) is satisfied in particular when the prescribed Ricci form, \(\tilde{\varrho}\), can be locally realised by H\"older continuous potentials, which is one of the conditions is Theorem \ref{thm-1-1} part (a).


\begin{thebibliography}{99}
%\bibitem {al-linear} Aleyasin, S. A. \emph{A note on  linear elliptic theory for edge-cone K\"ahler metrics}, 

%\bibitem {ar-de-la} Arezzo, C. della Vedova, A. La Nave, G. \emph{On the curvature of conic K\"ahler-Einstein metrics},
%available on the the authors' web-pages, 
%\url{http://www.math.uiuc.edu/~lanave/coniccurvaturefinal2.pdf}

\bibitem {au} Aubin, T. \emph{\'Equations du type Monge-Amp\`ere sur les vari\'et\'es k\"ahl\'eriennes compactes}, Bull. Sci. Math. (2) {\bf 102} (1978), no. 1, 63–95


\bibitem {br} Brendle, S. \emph{Ricci-flat K\"ahler metrics with edge singularities
} Int. Math. Res. Notices 2012, 1-40. 



\bibitem{ca-zh} Calamai, S. Zheng, K. \emph{Geodesics in the space of K\"ahler cone metrics},Amer. J. Math.
137, No. 5, 
1149-1208 (2015).

\bibitem {ca-gu-pa} Campana, F. Geunancia, H. P\v{a}un, M. \emph{Metrics with cone singularities along normal crossing divisors and holomorphic tensor fields}, Ann. Scient. \'Ec. Norm. Sup. \(4^e\), t. 46, 2013, 879-916


\bibitem {ch-do-su-1} Chen, X-X. Donaldson, S. Sun, S. \emph{K\"ahler-Einstein metrics and stability}, and \emph{K\"ahler-Einstein metrics on Fano manifolds I. }, 
arXiv:1210.7494,
J. Amer. Math. Soc. {\bf 28} (2015), 183-197 



\bibitem {ch-do-su-2} Chen, X-X. Donaldson, S. Sun, S. \emph{K\"ahler-Einstein metrics on Fano manifolds II }, 
  J. Amer. Math. Soc. {\bf 28} (2015), 199-234,


\bibitem {ch-do-su-3} Chen, X-X. Donaldson, S. Sun, S.  \emph{K\"ahler-Einstein metrics on Fano manifolds  III. }, 
  J. Amer. Math. Soc. {\bf 28} (2015), 235-278 
  
\bibitem {ch-wa} Chen, X-X. Wang, Y. \emph{\(C^{2,\alpha}\)
	-estimate for Monge-Amp\`ere equations with H\"older-continuous right hand side}
Ann. Global Analysis Geom.
Volume 49, {\bf 2} (2016), 195–204.

\bibitem{chu} Chu, J. 
\emph{ \(C^{2,\alpha}\) regularities and estimates for nonlinear elliptic and parabolic equations in geometry
	}
Calc. Var. Partial Diff. Eq.
(2016), 55:8

\bibitem {do} Donaldson, S. K. \emph{K\"ahler metrics with cone singularities along a divisor}, in Essays in mathematics and its applications, 49-79, Springer, Heidelberg, 2012.


\bibitem {gu-pa} Geunancia, H. P\v{a}un \emph{Conic singularities metrics with prescribed Ricci curvature: the case of general cone angles along normal crossing divisors}, 
    J. Diff. Geom.
    Volume 103, {\bf 1} (2016), 15-57.

\bibitem {gi-tr} Gilbarg, D. Trudinger, N. \emph{Elliptic partial differential equations of second order}, {\bf 224}, Grundlehren der mathematischen Wissenschaften, Springer Verlag, 1997.

\bibitem {je-ma-ru} Jefferes T, T. Mazzeo, R. Rubinstein, Y.  \emph{K\"ahler-Einstein metrics with edge singularities},
Annals of Math. 183 (2016), 95-176.



\bibitem {ko} Ko\l odziej, S. {\it H\"older continuity of solutions to the complex Monge-Amp\`ere equation with the right hand side in \(L^p\): the case of compact K\"ahler manifolds}, Math. Ann. {\bf 342} (2008), no. 1, 379-386.


\bibitem {pa} P\v{a}un, M. \emph{Regularity properties of the degenerate Monge-Amp\`ere equation on compact K\"ahler manifolds}, Chin. Ann. Math. {\bf 29B} (6), 2008, 623-630

\bibitem{rub} Rubinstein, Y. 
\emph{
Smooth and singular K\"ahler-Einstein metrics in: Geometric and Spectral Analysis}
 P. Albin et al., Eds., Contemp. Math. 630, AMS and Centre Recherches Mathematiques, 2014, pp. 45-138.

\bibitem {si} Siu, Y-T. \emph{Lectures on Hermite-Einstein metrics for stable bundles and K\"ahler-Einstein metrics}, Seminare der Deutschen Mathematiker-Vereinigung, Band {\bf 8}, Birkh\"auser, 1987.

\bibitem {ti} Tian, G. \emph{ K-stability and K\"ahler-Einstein metrics}, 
 arXiv:1211.4669v2

\bibitem {ti-ya}  Tian, G. Yau, S-T. \emph{Complete K\"ahler manifolds with zero zero Ricci curvature. I} Jour. Amer. Math. Soc., {\bf 3}, (1990), no. 3, 579-609

\bibitem {yao} Yao, C. \emph{Existence of weak conical K\"ahler–Einstein metrics along smooth hypersurfaces}, Math. Ann. (2014)

\bibitem {ya} Yau, S-T. \emph{On the Ricci curvature of a compact K\"ahler manifold and the complex Monge-Amp\`ere equation. I}, Comm. Pure Appl. Math. {\bf 31} (1978), no. 3, 339–411.

\end{thebibliography}
\end{document}